\documentclass[11pt,letter paper,twoside,openright,final]{article} 
\usepackage{amsfonts}
\usepackage[utf8]{inputenc}
\usepackage{amssymb,amsmath,amsthm}
\usepackage{mathrsfs}
\usepackage{amsthm}
\usepackage{tabularx}
\usepackage[pdftex]{graphicx}
\usepackage{float}
\usepackage{multicol}
\usepackage{verbatim}
\usepackage[dvipsnames]{xcolor}
\usepackage[pdftex]{hyperref}
\usepackage{enumitem}
\usepackage{adjustbox}
\usepackage[inner=3cm, outer= 3cm, top=3cm, bottom=3cm]{geometry}

\usepackage[backref=true, style = numeric, maxcitenames=3, maxbibnames=100, backend=biber]{biblatex}
\DeclareFieldFormat{pages}{#1} 
\DeclareListFormat{author}{\printnames{author}}
\DeclareFieldFormat[article]{title}{#1, \addcomma} 
\DeclareFieldFormat[inproceedings]{title}{#1} 
\DeclareFieldFormat[unpublished]{title}{#1} 
\DeclareFieldFormat[thesis]{title}{\textit{#1}} 
\DeclareFieldFormat[article]{volume}{\textbf{#1}} 
\DeclareFieldFormat{note}{#1} 
\DeclareFieldFormat{doi}{#1} 
\DeclareListFormat{publisher}{#1}
\hypersetup{colorlinks, citecolor=black, filecolor=black, linkcolor=black, urlcolor=black, pdftex}
\newcolumntype{C}[1]{>{\centering\arraybackslash}m{#1}}
\theoremstyle{definition}

\theoremstyle{definition}
\newtheorem{theorem}{Theorem}[section]
\newtheorem{lemma}[theorem]{Lemma}
\newtheorem{corollary}[theorem]{Corollary}

\newtheorem{proposition}[theorem]{Proposition}
\newtheorem{problem}[theorem]{Problem}

\newcommand{\Cay}{\mathrm{Cay}}

\def \CC {\mathbb {C}}
\def \ZZ {\mathbb {Z}}

\newcommand{\modulo}[3]{#1 \equiv #2 \: (\textrm{mod }#3)}
\newcommand{\notmodulo}[3]{#1 \not\equiv #2 \: (\textrm{mod }#3)}
\newcommand{\modul}[2]{#1 \: (\textrm{mod }#2)}

\newcommand\blfootnote[1]{%
  \begingroup

  \renewcommand\thefootnote{}\footnote{#1}%

  \addtocounter{footnote}{-1}%

  \endgroup

}

\DeclareBibliographyDriver{article}{%
  \usebibmacro{bibindex}%
  \usebibmacro{begentry}%
  \printnames{author} \newunit\addcomma
  \printfield{title} \newunit\addcomma
  \printfield{journaltitle}  
  \printfield{volume}%
  \iffieldundef{number}{}{\printtext[parens]{\printfield{number}}}\newunit
  \printfield{note} \setunit{\addspace}
  \iffieldundef{year}{\addcomma\printfield{doi}}{\printtext[parens]{\printfield{year}} \newunit\addcomma} 
  \printfield{pages}
  \usebibmacro{finentry}}

\DeclareBibliographyDriver{book}{%
  \usebibmacro{bibindex}%
  \usebibmacro{begentry}%
  \printnames{author} \newunit\addcomma
  \printfield{title} \newunit\addcomma
  \printlist{publisher}%
  \iflistundef{location}{}{\newunit\addcomma\printlist{location}} 
  \printtext[parens]{\printfield{year}} 
  \iffieldundef{pages}{\setunit{\adddot}}{\newunit\addcomma 
  \printfield{pages}}
  \usebibmacro{finentry}}

\DeclareBibliographyDriver{thesis}{%
  \usebibmacro{bibindex}%
  \usebibmacro{begentry}%
  \printnames{author} \newunit\addcomma
  \printfield{title} \newunit\addcomma
  \printfield{type} \newunit\addcomma
  \printlist{institution} \newunit\addcomma
  \printlist{location} \newunit\addcomma
  \printfield{year}
  \usebibmacro{finentry}}

\DeclareBibliographyDriver{inproceedings}{%
  \usebibmacro{bibindex}%
  \usebibmacro{begentry}%
  \printnames{author} \newunit\addcomma
  \printfield{title} \newunit\addcomma
  {\printtext{in$\colon$}\printnames{editor} \printtext[parens]{Eds.}} \newunit\addcomma
  \printfield{booktitle} \newunit\addcomma
  \printlist{publisher} \setunit{\addspace}
  \iflistundef{location}{}{\addcomma\printlist{location}}
  \printtext[parens]{\printfield{year}} \newunit\addcomma 
  \printfield{pages} 
  \usebibmacro{finentry}}

\DeclareBibliographyDriver{unpublished}{%
  \usebibmacro{begentry}%
  \printnames{author} \newunit\addcomma
  \printfield{title} \newunit\addcomma
  \printfield{note} 
  \usebibmacro{finentry}}

\begin{document}

\begin{center}
\Large{\textbf{Distance magic labelings of Cartesian products of cycles}} \\ [+4ex]
Ksenija Rozman{\small$^{a,*}$}, Primo\v z \v Sparl{\small$^{a, b, c}$} \\ [+2ex]
{\it \small
$^a$Institute of Mathematics, Physics and Mechanics, Ljubljana, Slovenia\\
$^b$University of Ljubljana, Faculty of Education, Ljubljana, Slovenia\\
$^c$University of Primorska, Institute Andrej Maru\v si\v c, Koper, Slovenia\\
}
\end{center}


\blfootnote{

Email addresses:
ksenija.rozman@pef.uni-lj.si (Ksenija Rozman),
primoz.sparl@pef.uni-lj.si (Primo\v z \v Sparl)

* - corresponding author at: Institute of Mathematics, Physics and Mechanics, Ljubljana, Slovenia
}

\hrule

\begin{abstract}

A graph of order $n$ is {\em distance magic} if it admits a bijective labeling of its vertices with integers from $1$ to $n$ such that each vertex has the same sum of the labels of its neighbors. In this paper we classify all distance magic Cartesian products of two cycles, thereby correcting an error in a widely cited paper from 2004. Additionally, we show that each distance magic labeling of a Cartesian product of cycles is determined by a pair or quadruple of suitable sequences, thus obtaining a complete characterization of all distance magic labelings of these graphs. We also determine a lower bound on the number of all distance magic labelings of $C_{m} \square C_{2m}$ with $m \ge 3$ odd.

\end{abstract}

\hrule

\begin{quotation}


\end{quotation} 

\section{Introduction} %

A graph $\Gamma=(V,E)$ of order $n$ is said to be {\em distance magic} if it admits a bijection $\ell: V \to \{1,\ldots, n\}$ such that the sum of the labels of all neighbors of a vertex $v \in V$, called the {\em weight} of $v$, is the same for all $v \in V$. The concept of distance magic graphs has been motivated by the notion of magic rectangles at the beginning of this century and has since been studied by numerous researchers, as indicated by a survey of existing results on distance magic labelings gathered in~\cite{AFK11} (see also~\cite{Gallian22} for a dynamic survey of all graph labelings). 

One of the first families studied was the family of Cartesian products of cycles. The necessary and sufficient condition from~\cite[Theorem 2.1]{RSP04} for such graphs to be distance magic is one of the most cited results on distance magic graphs in the literature (with more than 30 citations to date). However, while the condition given is indeed sufficient, it is unfortunately not necessary. In this paper we thus revisit the problem of determining all distance magic Cartesian products of cycles and provide a complete classification. More precisely, we prove the following theorem.

\begin{theorem} \label{theorem1}
Let $m,n$ be integers with $n \ge m \ge 3$. Then the Cartesian product of cycles $C_m \square C_n$ is distance magic if and only if $n=2m$ with $m$ odd, or $n=m$ with $\modulo{m}{2}{4}$.
\end{theorem}

In addition, we address the open problem from~\cite{RSP04} of finding all distance magic labelings of Cartesian products of cycles. In particular, for each of the two possibilities from Theorem~\ref{theorem1} we characterize all distance magic labelings of such graphs in terms of the existence of appropriate sequences of integers (see Theorem~\ref{lemma2} and Theorem~\ref{lemma22}). Moreover, we establish a lower bound for the number of all distance magic labelings of $C_m \square C_{2m}$ with $m \ge 3$ odd (see Proposition~\ref{lemma4}). 

The main tool in our arguments is a result from~\cite{MS21} linking the property of being distance magic (for regular graphs) to eigenvalues and eigenvectors, and the irreducible group characters of abelian groups, which turn out to be highly useful in the study of distance magic Cayley graphs (see for instance,~\cite{FMMR23,MS21,MS23, MSip}).

\section{Preliminaries} \label{sec:2}
In this section we state some basic definitions and notations that will be used throughout this paper. All graphs are assumed to be simple, finite, connected and undirected. 

For an integer $n$ we let $\ZZ_n$ denote the ring of residue classes modulo $n$ as well as the cyclic group of order $n$. It is well known that the Cartesian product $C_m \square C_n$ with $m,n \ge 3$ is in fact the tetravalent Cayley graph of the abelian group $\ZZ_m \times \ZZ_n$ with the connection set $S= \{\pm (1,0), \pm(0,1)\}$. Recall that the Cayley graph $\Cay(A;S)$ of an additive abelian group $A$ with respect to the connection set $S$, where $S \subset A$ with $S=-S$ and $0 \notin S$, is the graph with vertex-set $A$ in which two vertices $g,h \in A$ are adjacent if and only if $h-g \in S$. 

As in~\cite{RSup} we use a slightly nonstandard definition of a distance magic labeling which however leads to the same definition of a graph being distance magic (see~\cite{RSup}). For a positive integer $n$ we let $\mathcal{N}_n = \{1-n, 3-n, 5-n, \ldots, n-1\}$. Let $\Gamma =(V, E)$ be a tetravalent graph of order $n$. Then a {\em distance magic labeling} of~$\Gamma$ is a bijection $\ell:V \longrightarrow \mathcal{N}_n$ such that the weight of $v$ is equal to $0$ for all $v \in V$. Throughout this paper we will always be working with this alternative definition of a distance magic labeling.

\subsection{The characters}

As was first pointed out in~\cite{MS21} the property of being distance magic for regular graphs can be expressed in terms of eigenvalues and corresponding eigenvectors of the adjacency matrix. In particular, the following holds.

\begin{lemma}~\cite[Lemma 2.1]{MS21} \label{lemma:chi}
Let $\Gamma=(V,E)$ be a regular graph of order $n$ and even valency. Then $\Gamma$ is distance magic if and only if $0$ is an eigenvalue of the adjacency matrix of~$\Gamma$ and there exists a corresponding eigenvector with the property that a certain permutation of its entries results in the arithmetic sequence 
$1-n, 3-n, 5-n, \ldots, n-1.$
In particular, if $\Gamma$ is distance magic then $0$ is an eigenvalue for the adjacency matrix of~$\Gamma$ and there exists a corresponding eigenvector whose entries are pairwise distinct.
\end{lemma}

We remark that in fact the proof of Lemma~2.1 reveals that each distance magic labeling of a regular graph corresponds to an eigenvector for the eigenvalue 0 of the adjacency matrix in the sense of Lemma~\ref{lemma:chi}. Indeed, if $\ell $ is a distance magic labeling of a regular graph~$\Gamma$, the $v$-entry of a corresponding eigenvector can simply be set as $\ell(v)$ and vice versa. 

In the case of Cayley graphs one can use the theory of irreducible group characters to determine when a suitable eigenvector for the eigenvalue $0$ exists. As was successfully demonstrated in~\cite{FMMR23, MS21, MS23, MSip} this approach is particularly fruitful in the case of Cayley graphs of abelian groups. Here we review some definitions and results important for our situation but refer the reader to~\cite{MS21} for details.

Consider the abelian group $A = \ZZ_{m} \times \ZZ_{n}$ for some integers $m, n \ge 3$. Following~\cite{Alperin95}, the group of irreducible characters of the group $A$ consists of the $mn$ homomorphisms $\chi_{(a,b)}: A \to \{z \in \CC \colon z^{mn}=1\}$, $(a,b) \in A$, defined by 
\begin{equation} \label{eq:chi}
\begin{split}
\chi_{(a,b)}\left((x,y)\right) & = \left( e^{\frac{2 \pi i}{m}} \right)^{ax} \cdot \left( e^{\frac{2 \pi i}{n}} \right)^{by} \\
\end{split}
\end{equation} 
for all $(x,y) \in A$. Moreover, by~\cite{Godsil93} (see also~\cite{FMMR23, MS21}), the spectrum of the Cayley graph $\Cay(A;S)$ is 
\begin{equation} \label{eq:zero}
\{ \chi_{(a,b)}(S) \colon (a,b) \in A \}, \quad \textrm{ where } \quad  \chi_{(a,b)}(S) = \sum_{s \in S} \chi_{(a,b)}(s).
\end{equation}
The following fact was also pointed out in~\cite{MS21} (see also~\cite[Lemma 9.2]{Godsil93}). For the irreducible character $\chi$, let $\bold{w}_\chi$ denote the column vector indexed by the elements of $A$, with the $(x,y)$-entry of $\bold{w}_\chi$ equal to $\chi((x,y))$. Then the $mn$ vectors $\bold{w}_{\chi_{(a,b)}}$ where $(a,b) \in A$, form a complete set of eigenvectors for the adjacency matrix of $\Cay(A;S)$ and $\bold{w}_{\chi_{(a,b)}}$ corresponds to the eigenvalue $\chi_{(a,b)}(S)$. Therefore Lemma~\ref{lemma:chi} implies that for $\Cay(A;S)$ to be distance magic, it must admit at least one pair $(a,b) \in A$ such that $\chi_{(a,b)}(S)=0$. This leads to the following definition which was first introduced in~\cite{MS23}. Let $m,n \ge 3$ be integers and let $S \subset A$ be a subset such that $S=-S$ and $0 \notin S$. The set 
$$\mathscr{A}_{m,n}(S) = \{(a,b) \in A \colon \chi_{(a,b)}(S)=0\}$$
is said to be the set of all {\em admissible} elements of $A$ for $S$. 

\subsection{A necessary condition}

As already mentioned we will be studying the distance magic property of Cartesian products of two cycles. Recall that for $n \ge m \ge 3$ the Cartesian product  $\Gamma= C_m \square C_n$ is isomorphic to the Cayley graph $\Cay(\ZZ_m \times \ZZ_n; S)$, where $S= \{\pm (1,0), \pm(0,1)\}$. 

In this section we provide a short proof that the condition on $m$ and $n$ from Theorem~\ref{theorem1} is indeed necessary for $\Gamma$ to be distance magic. We do so using Lemma~\ref{lemma:chi} and irreducible group characters of $\ZZ_m \times  \ZZ_n$. 

Let  $n \ge m \ge 3$ and let $\Gamma=\Cay(\ZZ_m \times \ZZ_n;S)$, where $S=\{\pm (1,0), \pm(0,1)\}$. Observe that by~(\ref{eq:chi}) we have that $\chi_{(a,b)}((0,0))=1$ for all $(a,b) \in \ZZ_m \times \ZZ_n$. Moreover, (\ref{eq:chi}) and (\ref{eq:zero}) imply that $\chi_{(a,b)}(S)=0$ if and only if 
\begin{equation*}
e^{\frac{2\pi i a}{m}} +e^{-\frac{2\pi i a}{m}}+e^{\frac{2\pi i b}{n}}+e^{-\frac{2\pi i b}{n}}= 2 \cos \left( \frac{2 \pi a}{m}\right) + 2 \cos \left( \frac{2 \pi b}{n}\right)=0,
\end{equation*}
 which is equivalent to 
\begin{equation} \label{eq:nec}
\frac{2b}{n} \pm \frac{2a}{m}= 1+2k
\end{equation}
for some integer $k$. Note that this implies that if $\Gamma$ is distance magic then at least one of $n,m$ must be even. Since the connection set $S$ will always be $\{\pm (1,0), \pm(0,1)\}$ we abbreviate $\mathscr{A}_{m,n}(S)$ by~$\mathscr{A}_{m,n}$. 

\begin{proposition} \label{lemma:nec}
Let $n \ge m \ge 3$ and let $\Gamma=\Cay(\ZZ_m \times \ZZ_n;\{\pm (1,0), \pm(0,1)\})$. If $\Gamma$ is distance magic then $n=2m$ with $m$ odd, or $n=m$ with $\modulo{m}{2}{4}$.
\end{proposition}

\begin{proof}
Suppose that $\Gamma$ is distance magic. Letting $d=\gcd(m,n)$ we set $n_0$ and $m_0$ to be the coprime positive integers such that $n=d n_0$ and $m=d m_0$. Then (\ref{eq:nec}) is equivalent to $2bm_0 \pm 2an_0 = d m_0 n_0(1+2k)$ for some integer $k$.
Suppose that there exists an odd prime divisor $p$ of $m_0$. Then for each $(a,b)\in \mathscr{A}_{m,n}$ the prime $p$ divides $an_0$, and so $\gcd(m_0, n_0)=1$ implies that $p$ divides $a$. By~(\ref{eq:chi}) we thus have that $\chi_{(a,b)}((\frac{m}{p}, 0))=1=\chi_{(a,b)}((0,0))$ for all $(a,b) \in \mathscr{A}_{m,n}$. Since each eigenvector for the eigenvalue~$0$ thus has equal entries at $(\frac{m}{p}, 0)$ and $(0,0)$, this contradicts Lemma~\ref{lemma:chi}. In a similar way we show that there are no odd prime divisors of $n_0$. Therefore, since $n \ge m$, we have that $m_0=1$ and $n_0=2^t$ for some integer $t \ge 0$, which implies that $m=d$ and $n=m2^t$.

We first consider the possibility that $t \ge 1$ and prove that in this case $m$ is odd and $t=1$ (and thus $n=2m$). Since $t \ge1$, (\ref{eq:nec}) is equivalent to $b \pm 2^t a= 2^{t-1} m (1+2k)$, $k \in \ZZ$. Therefore, if $t \ge 2$ or $m$ is even, $b$ must be even, and so $\chi_{(a,b)}((0,\frac{n}{2}))=1=\chi_{(a,b)}((0,0))$ for all $(a,b) \in \mathscr{A}_{m,n}$, again contradicting Lemma~\ref{lemma:chi}. Therefore, $t=1$ and $m$ is odd, as claimed.

We now consider the possibility that $t=0$, that is, when $n=m$. In this case (\ref{eq:nec}) is equivalent to $b \pm a = \frac{m(1+2k)}{2}$ for some $k \in \ZZ$, which implies that $m$ is even. If $m$ is divisible by 4 then $a$ and $b$ are of the same parity and then $\chi_{(a,b)}((\frac{m}{2},\frac{m}{2}))=1$ for all $(a,b) \in \mathscr{A}_{m,m}$, contradicting Lemma~\ref{lemma:chi}. 
Therefore, if $\Gamma$ is distance magic then $n=2m$ with $m$ odd or $n=m$ with $\modulo{m}{2}{4}$.
\end{proof}

To complete the proof of Theorem~\ref{theorem1} we only need to show that for each odd integer $m \ge 3$ the graphs  $\Gamma_1 = C_m \square C_{2m}$ and $\Gamma_2 = C_{2m} \square C_{2m}$ are indeed distance magic. For~$\Gamma_1$ this is done in Section~\ref{sec:3} and for~$\Gamma_2$ in Section~\ref{sec:4}. In fact, in Section~\ref{sec:3} we describe all distance magic labelings of~$\Gamma_1$ and in Section~\ref{sec:4} all distance magic labelings of~$\Gamma_2$, thereby addressing the problem posed in~\cite{RSP04} of determining all such labelings. It turns out that the following characterization of all admissible elements of the corresponding groups and the ensuing Corollary~\ref{lemma:subm2m} and Corollary~\ref{lemma:submm} will be of use.

Consider the case of~$\Gamma_1$. In this case (\ref{eq:nec}) is equivalent to $b \pm 2a =m (1+2k)$, $k \in \ZZ$. Since $0 \le a < m$ and $0 \le b < 2m$, we have that $-2m < b \pm 2a < 4m$, which in turn implies that $k \in \{0, \pm 1\}$. Moreover, if $0 \le a \le \frac{m-1}{2}$ then $k=0$, while if $\frac{m+1}{2} \le a < m$, then $k = \pm 1$. It is now easy to see that 
\begin{equation}\label{eq:admissiblem2m}
\mathscr{A}_{m, 2m} = \{(a, b) \colon 0 \le a <m, \; 0 \le b < 2m, \; b \equiv \modul{(m \mp 2a)}{2m}\}.
\end{equation}
In a similar way, considering the case of $\Gamma_2$, one can show that
\begin{equation}\label{eq:admissiblemm}
\mathscr{A}_{2m, 2m} = \{(a, b) \colon 0 \le a,b < 2m,\; b \equiv \modul{(m \mp a)}{2m}\}.
\end{equation}

Before stating and proving Corollary~\ref{lemma:subm2m} and Corollary~\ref{lemma:submm} we introduce the following notation and terminology that will be used in the remainder of this paper. We will represent a labeling $\ell$ of the vertex-set of a graph $\Gamma = C_m \square C_n$ by the $m \times n$ table $L=[\ell_{i,j}]^{0\le i < m}_{0 \le j < n}$ with $m$ rows and $n$ columns, where the entry $\ell_{i,j}$ corresponds to $\ell((i,j))$. We make an agreement that for all $m \times n$ tables in this paper the first index is always computed modulo~$m$, while the second one is always computed modulo~$n$. 

\begin{corollary} \label{lemma:subm2m}
Let $m \ge 3$ be an odd integer and suppose that $\Gamma=C_m \square C_{2m} \cong \Cay(\ZZ_m \times \ZZ_{2m}; \{\pm (1,0), \pm(0,1)\})$ is distance magic with a distance magic labeling~$\ell$. Then $\ell_{i,j}=-\ell_{i,j+m}$ for all $i,j$ with $0 \le i < m$, $0 \le j < 2m$.
\end{corollary}
\begin{proof}
By Lemma~\ref{lemma:chi} and the remark following it, $\ell$ corresponds to an eigenvector $\bold{w}$ for the eigenvalue 0 satisfying Lemma~\ref{lemma:chi}. Denote the $(x,y)$-entry of $\bold{w}$ by $\bold{w}_{x,y}$ and recall that
$$\ell_{x,y} = \bold{w}_{x,y}= \mkern-20mu\sum_{(a,b)\in \mathscr{A}_{m,2m}}  \mkern-20mu \lambda_{(a,b)} \chi_{(a,b)}((x,y))$$ 
for some $\lambda_{(a,b)}  \in \CC$. 

Let $(a,b) \in \mathscr{A}_{m,2m}$. By~(\ref{eq:chi}) we have that $\chi_{(a,b)}((0,0))=1$ and since by~(\ref{eq:admissiblem2m}) each~$b$ is odd, we also have that $\chi_{(a,b)}((0,m))=-1$. Therefore, 
\begin{equation*}
\chi_{(a,b)}((0,0)) + \chi_{(a,b)}((0,m)) = 0 \text{ for all } (a,b) \in \mathscr{A}_{m,2m}.
\end{equation*}
Consequently, for any $i,j$ with $0 \le i <m$, $0 \le j < 2m$, the fact that $\chi_{a,b}$ is a homomorphism implies that
\begin{equation*} 
\begin{split}
\ell_{i,j} + \ell_{i,j+m}& =  \bold{w}_{i,j} + \bold{w}_{i,j+m} = \mkern-20mu \sum_{(a,b) \in \mathscr{A}_{m, 2m}} \mkern-20mu \lambda_{(a,b)} \left(\chi_{(a,b)}((i,j)) + \chi_{(a,b)}((i,j+m))\right) =\\
&= \mkern-15mu \sum_{(a,b) \in \mathscr{A}_{m, 2m}} \mkern-15mu \lambda_{(a,b)} \chi_{(a,b)}((i,j)) \left(\chi_{(a,b)}((0,0)) + \chi_{(a,b)}((0,m))\right) = 0.
\end{split}
\end{equation*}
\end{proof}

\begin{corollary} \label{lemma:submm}
Let $m \ge 3$ be an odd integer and suppose that $\Gamma=C_{2m} \square C_{2m}= \Cay(\ZZ_{2m} \times \ZZ_{2m}; \{\pm (1,0), \pm(0,1)\})$ is distance magic with a distance magic labeling~$\ell$. Then $\ell_{i,j}=-\ell_{i+m,j+m}$ for all $i,j$ with $0 \le i,j < 2m$.
\end{corollary}

\begin{proof}
The proof is similar and is left to the reader (note that in this case (\ref{eq:admissiblemm}) implies that $a$ and $b$ are of different parity).
\end{proof}

\section{Distance magic labelings of $\bold{C_m \square C_{2m}}$ with $\bold{m \ge 3}$ odd} \label{sec:3}
Before we present a distance magic labeling for the graph $C_m \square C_{2m}$ with $m \ge 3$ odd, we introduce the terminology we will be using in this section. Let $m \ge 3$ be an odd integer and let $L = [\ell_{i,j}]^{0\le i < m}_{0 \le j < 2m}$ be an $m \times 2m$ table of integers. Then the {\em reduced table re(L)} of $L$ is the $m \times m$ table $[t_{i,j}]_{0 \le i,j < m}$ where
\begin{equation} \label{eq:re}
t_{i,j}=\ell_{i, 2j}.
\end{equation}

\noindent
Conversely, let $T=[t_{i,j}]_{0 \le i,j < m}$ be an $m \times m $ table of integers. Then the corresponding {\em extended table} $ex(T)$ of $T$ is the $m \times 2m$ table $[\ell_{i,j}]^{0\le i < m}_{0 \le j < 2m}$, where
\begin{equation} \label{eq:l}
\begin{split}
\ell_{i,j}=\left\{
\begin{array}{ll}
t_{i, \frac{j}{2}}, & j \textrm{ even} \\
-t_{i, \frac{j+m}{2}}, & j \textrm{ odd}.
\end{array}
\right.
\end{split}
\end{equation}
Note that since $m$ is odd, the above definition~(\ref{eq:l}) does indeed make sense and implies that $\ell_{i,j}=-\ell_{i, j+m}$ for all $i,j$ with $0 \le i<m$, $0 \le j < 2m$. Observe that for any $m \times m$ table~$T$ we have that $re(ex(T))=T$ and similarly, if $C_m \square C_{2m}$ is distance magic and $L$ is the $m \times 2m$ table corresponding to a distance magic labeling of this graph, then $ex(re(L))=L$ by Corollary~\ref{lemma:subm2m}. 

Let $m \ge 3$ be an odd integer. An $m \times m$ table $T$ is said to be {\em distance magic} if 
\begin{equation} \label{eq:nn}
\textrm{exactly one of } k, -k \textrm{ appears in } T \textrm{ for all } k \in \mathcal{N}_{2m^2}
\end{equation} and
\begin{equation} \label{eq:t}
 t_{i-1, j} + t_{i+1, j} = t_{i, j+\frac{m-1}{2}} + t_{i, j-\frac{m-1}{2}} \ \textrm{ for all } i,j \textrm{ with } 0 \le i,j <m.
\end{equation}
The following observation indicates why we call such tables distance magic.

\begin{lemma} \label{welldefined}
Let $m \ge 3$ be an odd integer and let $\Gamma = C_m \square C_{2m}$. If $\Gamma$ is distance magic and $L$ denotes the $m \times 2m$ table corresponding to a distance magic labeling of $\Gamma$ then $re(L)$ is distance magic. Conversely, if an $m \times m$ table $T$ is distance magic then $ex(T)$ corresponds to a distance magic labeling of $\Gamma$ and consequently $\Gamma$ is distance magic. 
\end{lemma}

\begin{proof}
Assume first that $\Gamma$ is distance magic, let $L$ denote the $m \times 2m$ table corresponding to some distance magic labeling $\ell$ of $\Gamma$ and let $T=re(L)$. Then $L$ consists of the $2m^2$ (odd) integers from $\mathcal{N}_{2m^2}$. By Corollary~\ref{lemma:subm2m} we have that $\ell_{i,j}=-\ell_{i, j+m}$ for all $i,j$ with $0 \le i <m$, $0 \le j <2m$, and so $T$ clearly satisfies~(\ref{eq:nn}). Moreover, using~(\ref{eq:re}), Corollary~\ref{lemma:subm2m} and the fact that $\ell$ is a distance magic labeling of~$\Gamma$ we find that for all $i,j$ with $0 \le i,j < m$ 
$$ t_{i-1, j} + t_{i+1, j} = \ell_{i-1, 2j} + \ell_{i+1, 2j} = -\ell_{i, 2j-1} - \ell_{i, 2j+1} = \ell_{i, 2j+m-1} + \ell_{i, 2j+m+1} = t_{i, j+\frac{m-1}{2}} + t_{i, j-\frac{m-1}{2}}, $$
(note that $\modulo{\frac{m+1}{2}}{-\frac{m-1}{2}}{m}$) proving that $T$ is distance magic.

Conversely, suppose that $T$ is distance magic $m \times m$ table and let $L=ex(T)$. Since by definition for all odd $k \in \mathcal{N}_{2m^2}$, exactly one of $k, -k$ appears in $T$, (\ref{eq:l}) implies that $L$ consists of all odd $k \in \mathcal{N}_{2m^2}$. 
Pick $i, j$ with $0 \le i < m$, $0 \le j < 2m$. Then (\ref{eq:l}) and (\ref{eq:t}) imply that if $j$ is even, 
$$\ell_{i-1, j} + \ell_{i+1, j} + \ell_{i, j-1} + \ell_{i, j+1} = t_{i-1, \frac{j}{2}} + t_{i+1, \frac{j}{2}} - t_{i, \frac{j}{2}+\frac{m-1}{2}} - t_{i, \frac{j}{2}+\frac{m+1}{2}}=0$$ while if $j$ is odd
$$\ell_{i-1, j} + \ell_{i+1, j} + \ell_{i, j-1} + \ell_{i, j+1} = - t_{i-1, \frac{j+m}{2}} - t_{i+1, \frac{j+m}{2}}  + t_{i, \frac{j-1}{2}} + t_{i, \frac{j+1}{2}}=0. $$
Therefore, $L$ corresponds to a distance magic labeling of $\Gamma$.
\end{proof}

For integers $m,n \ge 3$ and $m \times n$ tables $T$ and $T'$ of integers we let the sum $T+T'$ of these two tables be the $m \times n$ table obtained by component-wise addition. Moreover, let $R=[r_0, r_1, \ldots, r_{m-1}]$ be a sequence of $m$ integers and let $s \in \{0,1, \ldots, m-1\}$. Then $\mathcal{T}(R,s)$ is the $m \times m$ table $[t_{i,j}]_{0 \le i,j <m}$, where $t_{i,j}=r_{j-si}$ (the index computed modulo $m$). Note that $\mathcal{T}(R,s)$ is the $m \times m$ table obtained by setting the 0-th row (the row consisting of all $t_{0,j}$ with $0 \le j < m$) to be $R$ and then obtaining each subsequent row by making a right cyclic $s$-shift of the previous row. If we set $X=\mathcal{T}(R,s)$ we denote the entries of such a table by $x_{i,j}$ with $0 \le i,j <m$.

\begin{proposition} \label{lemma1}
Let $m \ge 3$ be an odd integer and let $\Gamma=C_m \square C_{2m}$. Let $a_0, a_1, \ldots, a_{m-1} \in \{1-2m^2, \ldots, 2m^2-1\}$ be odd integers and $b_0, b_1, \ldots, b_{m-1} \in \{-2m^2, \ldots, 2m^2-2\}$ be even integers such that $a_0=1$, $b_0=0$ and that $\{x(a_i+b_j) \colon x \in \{-1, 1\}, 0 \le i,j <m  \} = \mathcal{N}_{2m^2}$. Then the table $\mathcal{T}([a_0, a_1, \ldots, a_{m-1}], \frac{m-1}{2})+ \mathcal{T}([b_0, b_1, \ldots, b_{m-1}], \frac{m+1}{2})$ is distance magic. Consequently, $\Gamma$ is distance magic. 
\end{proposition}

\begin{proof}
Let $ A=\mathcal{T}([a_0, a_1, \ldots, a_{m-1}], \frac{m-1}{2})$ and $B=\mathcal{T}([b_0, b_1, \ldots, b_{m-1}], \frac{m+1}{2})$. Let $i,j$ be integers with $0 \le i,j <m$. Then 
\begin{equation} \label{eq:sum_a}
a_{i-1,j}=a_{j-\frac{m-1}{2}(i-1)}=a_{i,j+\frac{m-1}{2}} \quad \textrm{and} \quad a_{i+1,j} = a_{j-\frac{m-1}{2}(i+1)} = a_{i,j-\frac{m-1}{2}},
\end{equation}
\begin{equation} \label{eq:sum_b}
b_{i-1,j}=b_{j-\frac{m+1}{2}(i-1)}=b_{i,j-\frac{m-1}{2}} \quad \textrm{and} \quad b_{i+1,j} = b_{j-\frac{m+1}{2}(i+1)} = b_{i,j+\frac{m-1}{2}}
\end{equation}
(recall that $\modulo{\frac{m-1}{2}}{-\frac{m+1}{2}}{m}$).

\noindent
Let $C=A+B=[c_{i,j}]_{0 \le i,j <m}$. From (\ref{eq:sum_a}) and (\ref{eq:sum_b}) we have that 
\begin{equation*}
\begin{split}
c_{i-1, j} + c_{i+1, j} = a_{i, j+\frac{m-1}{2}} + b_{i, j-\frac{m-1}{2}} + a_{i, j-\frac{m-1}{2}}+ b_{i, j+\frac{m-1}{2}} = \\
 a_{i, j+\frac{m-1}{2}} + b_{i, j+\frac{m-1}{2}} + a_{i, j-\frac{m-1}{2}}+ b_{i, j-\frac{m-1}{2}} = c_{i, j+\frac{m-1}{2}} + c_{i, j-\frac{m-1}{2}}
\end{split}
\end{equation*}
for any $0 \le i,j <m$, proving that $C$ satisfies~(\ref{eq:t}). We now prove that $C$ also satisfies~(\ref{eq:nn}). Since $\{x(a_i+b_j) \colon x \in \{-1, 1\}, 0 \le i,j <m  \} = \mathcal{N}_{2m^2}$, it clearly suffices to prove that $\{a_i+b_j \colon 0 \le i,j <m  \}= \{c_{i,j} \colon 0 \le i,j <m \}$.  Since one inclusion follows from the definition of~$C$ we only verify that for each $i,j$ with $0 \le i, j < m$ the sum $a_i+b_j$ belongs to $\{c_{i,j} \colon 0 \le i,j <m \}$. Indeed,
\begin{equation*}
 c_{i-j, i+\frac{m-1}{2}(i-j)}= a_{i+\frac{m-1}{2}(i-j)-\frac{m-1}{2}(i-j)} + b_{i+\frac{m-1}{2}(i-j)-\frac{m+1}{2}(i-j)}= a_i+b_j. 
\end{equation*}
Therefore, $C$ is distance magic, and so Lemma~\ref{welldefined} implies that $\Gamma$ is distance magic. 
\end{proof}

In order to show that $C_m \square C_{2m}$ is distance magic for all odd integers $m \ge 3$ it thus suffices to find some specific pairs of sets $\{a_0, a_1, \ldots, a_{m-1}\}$ and $\{b_0, b_1, \ldots, b_{m-1}\}$, satisfying the conditions stated in Proposition~\ref{lemma1}. In the remainder of this section we call such pairs of sets, as well as the pair $\langle [a_0, a_1, \ldots, a_{m-1}], [b_0, b_1, \ldots, b_{m-1}]\rangle$ of corresponding sequences, {\em admissible}.\\

\begin{proposition}\label{example}
Let $m\ge 3$ be an odd integer. The pair $\left< \bold{a}, \bold{b} \right>$, where 
\setlength{\belowdisplayskip}{7pt}  \setlength{\belowdisplayshortskip}{7pt} 
\setlength{\abovedisplayskip}{7pt}  \setlength{\abovedisplayshortskip}{-3pt}
$$\bold{a}=[1,3,5, \ldots, 2m-3, 2m-1] \quad \textrm{and} \quad \bold{b}=[0, 2m, 4m, \ldots, 2(m-2)m, 2(m-1)m],$$
is admissible. Consequently, the graph $C_m \square C_{2m}$ is distance magic.
\end{proposition}

\begin{proof}
For each $i$ with $0\le i < m$ we have that $a_i=1+2i$ and $b_i= 2im$. It is thus clear that $\{x(a_i+b_j) \colon x \in \{-1, 1\}, 0 \le i, j <m\}= \mathcal{N}_{2m^2}$ (see Figure~\ref{fig:m7} for the case $m=7$). Therefore, the pair $\left< \bold{a}, \bold{b} \right>$ is admissible, and so the graph $C_m \square C_{2m}$ is distance magic by Proposition~\ref{lemma1}.
\end{proof}
\begin{figure}[h] 
\centering
\includegraphics[scale=0.65, trim={1.5cm 12cm 0 1.8cm}, clip]{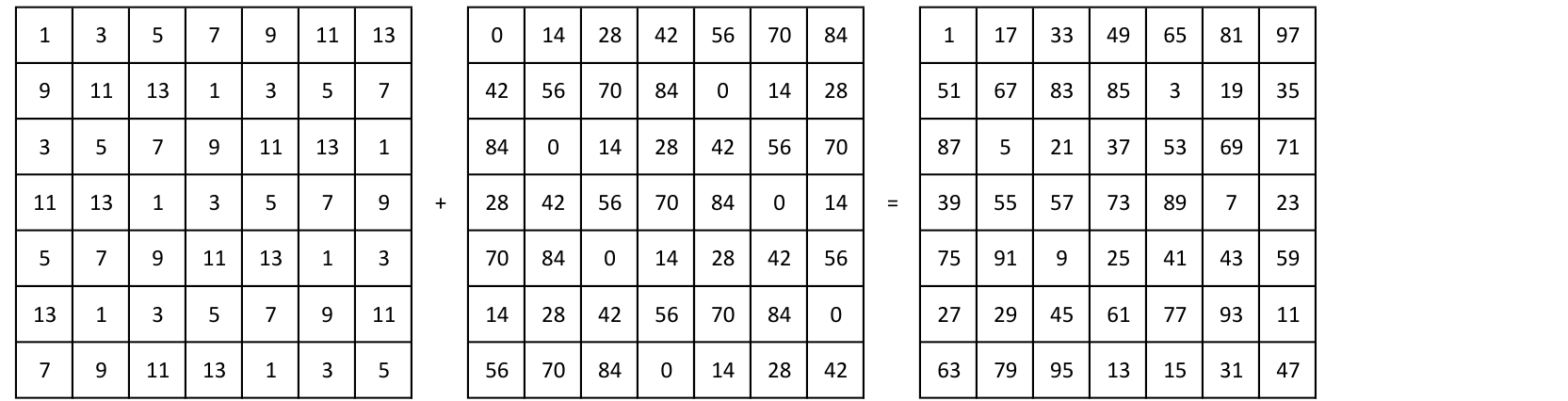}
\caption{Construction of a distance magic table for $m=7$.}
\label{fig:m7}
\end{figure}

As already mentioned, the authors of~\cite{RSP04} posed an open problem of determining all distance magic labelings of Cartesian products of cycles. In what follows we make a considerable step towards solving this problem for the case of $C_m \square C_{2m}$ where $m \ge 3$ is odd. By Lemma~\ref{welldefined} this is equivalent to finding all distance magic $m \times m$ tables. It turns out (see Theorem~\ref{lemma2}) that in fact each distance magic labeling of $C_m \square C_{2m}$ with $m \ge 3$ odd, arises from a suitable admissible pair. It thus suffices to determine all admissible pairs. This seems to be rather difficult and at the moment we are not able to determine the exact number of them. Nevertheless, we obtain a lower bound for their number (see Proposition~\ref{lemma4}) which seems to get closer and closer to the exact number as $m$ grows (see the remark at the end of this section).

\begin{lemma} \label{lemma:t2}
Let $m \ge 3$ be an odd integer and let $L$ be the $m \times 2m$ table corresponding to a distance magic labeling $\ell$ of $C_m \square C_{2m}$. Let $T=re(L)=[t_{i,j}]_{0 \le i,j < m}$. Then, for all $i,j$ with $0 \le i,j < m$, and all integers $s$ we have that
\begin{equation} \label{eq:t2}
 t_{i+1, j} - t_{i, j+\frac{m-1}{2}} = t_{i+1-s, j+\frac{m+1}{2}s} - t_{i-s, j+\frac{m-1}{2}+\frac{m+1}{2}s}.
\end{equation}
\end{lemma}

\begin{proof}
Since the indices are computed modulo~$m$ it clearly suffices to prove that (\ref{eq:t2}) holds for all non-negative integers~$s$. The equality~(\ref{eq:t2}) of course holds for $s=0$. Moreover, by Lemma~\ref{welldefined} the table $T$ is distance magic (implying that (\ref{eq:t}) holds), and so the equality~(\ref{eq:t2}) holds also for $s=1$. Now, assuming that (\ref{eq:t2}) holds for some $s \ge 1$ and all $i,j$ with $0 \le i,j <m$ we prove that it also holds for $s+1$ and all $i,j$ with $0 \le i,j <m$. Set $i'=i-s$ and $j'=j+ \frac{m+1}{2}s$. By~(\ref{eq:t2}) we have that 
\begin{equation*}
\begin{split}
 t_{i+1, j} - t_{i, j+\frac{m-1}{2}} &= t_{i'+1, j'} - t_{i', j'+\frac{m-1}{2}} = t_{i', j'+\frac{m+1}{2}} - t_{i'-1,j'} \\
& =  t_{i+1-(s+1), j+\frac{m+1}{2}(s+1)} - t_{i-(s+1), j+\frac{m-1}{2}+\frac{m+1}{2}(s+1)},
\end{split}
\end{equation*}
where we used (\ref{eq:t2}) for $s=1$ in the second equality. Therefore, (\ref{eq:t2}) holds for $s + 1$ and all $i,j$ with $0 \le i,j <m$, completing the induction step.
\end{proof}

\begin{theorem} \label{lemma2}
Let $m\ge 3$ be an odd integer and let $\ell$ be a distance magic labeling of $C_m \square C_{2m}$. If $\ell_{0,0}=1$ then $\ell$ corresponds to an admissible pair. Moreover, the corresponding pair of sequences is uniquely determined by $\ell$.
\end{theorem}

\begin{proof}
Let $C=re(L)=[c_{i,j}]_{0 \le i,j <m}$, where $L$ is the $m \times 2m$ table corresponding to~$\ell$, and let $\bold{a}$ and $\bold{b}$ be defined as follows:
\setlength{\belowdisplayskip}{7pt}  \setlength{\belowdisplayshortskip}{7pt} 
\setlength{\abovedisplayskip}{7pt}  \setlength{\abovedisplayshortskip}{-3pt}
\begin{equation*}
\bold{a} = [c_{i,\frac{m+1}{2}i} \colon 0 \le i < m],
\quad\quad
\bold{b} = [c_{m-i,\frac{m+1}{2}i}-1 \colon 0 \le i < m].
\end{equation*}
Letting $A=\mathcal{T}(\bold{a}, \frac{m-1}{2})$ and $B=\mathcal{T}(\bold{b}, \frac{m+1}{2})$ we claim that $A+B=C$. It clearly suffices to show that
\begin{equation} \label{eq:induction}
a_{i, \frac{m-1}{2}i+k} + b_{i, \frac{m-1}{2}i+k} = c_{i, \frac{m-1}{2}i+k}
\end{equation}
holds for all integers $i$, $k$ with $k \ge 0$. In fact, since
$$a_{i, \frac{m-1}{2}i+k} + b_{i, \frac{m-1}{2}i+k} = a_k + b_{k-i} = c_{k, \frac{m+1}{2}k} + c_{i-k, \frac{m-1}{2}i+\frac{m+1}{2}k}-1,$$
(\ref{eq:induction}) can be restated as
\begin{equation} \label{eq:induction2}
c_{k, \frac{m+1}{2}k} + c_{i-k, \frac{m-1}{2}i+\frac{m+1}{2}k}-1 = c_{i, \frac{m-1}{2}i+k}.
\end{equation}
We perform the proof that ($\ref{eq:induction2}$) indeed holds for all integers $i$ and $k$, with $k \ge 0$, by induction on~$k$. That $(\ref{eq:induction2})$ holds for $k=0$ and for all integers~$i$ follows from the fact that $c_{0,0}=1$. 
Next, suppose that (\ref{eq:induction2}) holds for some $k \ge 0$ and for all~$i$. We claim that for all~$i$ we have that
\begin{equation} \label{eq:induction3}
c_{k+1, \frac{m+1}{2}k +\frac{m+1}{2}} + c_{i-k-1, \frac{m-1}{2}i + \frac{m+1}{2}k+\frac{m+1}{2}} - 1 = c_{i, \frac{m-1}{2}i+k+1}.
\end{equation}
Setting $i'=i-1$ we see that the inductive hypothesis implies that 
\begin{equation*}
\begin{split}
 c_{i-k-1, \frac{m-1}{2}i + \frac{m+1}{2}k+\frac{m+1}{2}} - 1 &= c_{i'-k, \frac{m-1}{2}i'+\frac{m+1}{2}k}-1 \\
& = c_{i', \frac{m-1}{2}i'+k} - c_{k, \frac{m+1}{2}k} \\
& = c_{i-1, \frac{m-1}{2}i+k+\frac{m+1}{2}} - c_{k, \frac{m+1}{2}k}, 
\end{split}
\end{equation*}
and so ($\ref{eq:induction3}$) is equivalent to 
$$c_{k+1, \frac{m+1}{2}k +\frac{m+1}{2}} - c_{k, \frac{m+1}{2}k} = c_{i, \frac{m-1}{2}i+k+1} - c_{i-1, \frac{m-1}{2}i+k+\frac{m+1}{2}}. $$
Setting $i'=k$ and $j=\frac{m+1}{2}k +\frac{m+1}{2}$ and then applying Lemma~\ref{lemma:t2} for $s=i'+1-i$ we have that
\begin{equation*}
\begin{split}
c_{k+1, \frac{m+1}{2}k +\frac{m+1}{2}} - c_{k, \frac{m+1}{2}k} & =  c_{i'+1, j} - c_{i', j+ \frac{m-1}{2}} \\
& = c_{i,j+\frac{m+1}{2}i' + \frac{m+1}{2} - \frac{m+1}{2}i} - c_{i-1, j + \frac{m+1}{2}i' - \frac{m+1}{2}i } \\
& = c_{i, \frac{m-1}{2}i+k+1} - c_{i-1, \frac{m-1}{2}i+k+\frac{m+1}{2}},
\end{split}
\end{equation*}
which implies that (\ref{eq:induction3}) indeed holds for all~$i$ thus confirming that (\ref{eq:induction2}) (and hence (\ref{eq:induction})) holds for all integers~$i$, $k$ with $k \ge 0$.

By~(\ref{eq:induction}) each~$c_{i,j}$ is of the form $a_{i'}+b_{j'}$ for some $i', j'$. Since there are at most $m^2$ different sums of the form $a_{i'}+b_{j'}$ and (\ref{eq:nn}) holds for the table~$C$, it clearly follows that $\{x(a_{i'}+b_{j'})\colon x \in \{-1, 1\}, 0 \le i',j' <m  \}=\mathcal{N}_{2m^2}$, and so the pair $\langle \bold{a}, \bold{b} \rangle$ is admissible. 

We are now left with the proof that the above pair $\left< \bold{a}, \bold{b} \right>$ is the only admissible pair of sequences generating the table~$C$. Suppose that $\left< \bold{a'}, \bold{b'} \right>$ is an admissible pair of sequences generating~$C$. Let $A'=\mathcal{T}(\bold{a'}, \frac{m-1}{2})$ and $B'=\mathcal{T}(\bold{b'}, \frac{m+1}{2})$. Since $0$ appears in $B$ and in $B'$ at the positions $\{ \left( i,\frac{m+1}{2}i \right) \colon 0 \le i < m \}$, and these positions coincide with the entries $\{a_i \colon 0 \le i < m\}$ in~$A$, as well as with the entries $\{a'_i \colon 0 \le i < m\}$ in~$A'$, we have that $a_i=a'_i$ for all $0 \le i <m$. Similarly, since $1$ appears in $A$ and in $A'$ at the positions $\{ \left( i,\frac{m-1}{2}i \right) \colon 0 \le i < m \}$ and these positions coincide with the entries $\{b_{m-i} \colon 0 \le i < m\}$ in~$B$, as well as with the entries  $\{b'_{m-i} \colon 0 \le i < m\}$ in~$B'$, we have that $b_i=b'_i$ for all $0 \le i <m$. Therefore,  $\left< \bold{a},\bold{b} \right> =  \left< \bold{a'},\bold{b'} \right>$, implying that there exists a unique admissible pair of sequences $\left< \bold{a},\bold{b} \right>$ generating the table~$C$.
\end{proof} 

The next observation follows directly from the definition of an admissible pair.

\begin{lemma} \label{lemma:tilde}
Let $m \ge 3$ be an odd integer and let $\left<\bold{a},\bold{b}\right>$ be an admissible pair. Set 
\setlength{\belowdisplayskip}{7pt}  \setlength{\belowdisplayshortskip}{7pt} 
\setlength{\abovedisplayskip}{7pt}  \setlength{\abovedisplayshortskip}{-3pt}
$$ \bold{\tilde{a}}~=~[b_i+1 \colon 0 \le i <m] \quad \textrm{and} \quad \bold{\tilde{b}}=[a_i-1 \colon 0 \le i <m].$$ 
Then the pair $\langle\bold{\tilde{a}},\bold{\tilde{b}}\rangle$ is also admissible.
\end{lemma}

We proceed by introducing two specific admissible pairs $\left<\bold{a},\bold{b}\right>$ which will be called {\em basic pairs} in the rest of this section (see Lemma~\ref{lemma:type1} and Lemma~\ref{lemma:type2}). In addition, we introduce so-called {\em derived pairs}, obtained from basic pairs  $\langle \bold{a}, \bold{b}\rangle$ in such a way that one of $\bold{a}$ and $\bold{b}$ is fixed and the other is changed. In particular, {\em $\bold{a}$-derived pairs} are the pairs of the form $\left<\bold{a},\bold{b'}\right>$, where some of the entries of $\bold{b'}$ differ from those of $\bold{b}$ (in a specified way, see Lemma~\ref{lemma:type1} and Lemma~\ref{lemma:type2}). We define {\em $\bold{b}$-derived pairs} analogously.

\begin{lemma} 
\label{lemma:type1}
Let $m \ge 3$ be an odd integer and let 
\setlength{\belowdisplayskip}{7pt}  \setlength{\belowdisplayshortskip}{7pt} 
\setlength{\abovedisplayskip}{7pt}  \setlength{\abovedisplayshortskip}{-3pt}
$$\bold{a} = [1+2i \colon 0 \le i < m] \quad \textrm{and} \quad \bold{b} = [2im \colon 0 \le i < m].$$ 
Then the pair $\left< \bold{a},\bold{b} \right>$ is admissible. Moreover, if any number of elements $b_i \in \bold{b}$ other than $b_0$ are replaced by $b_i'=-b_i-2m$, the $\bold{a}$-derived pair $\left< \bold{a},\bold{b'} \right>$ is admissible. Similarly, if any number of elements $a_i \in \bold{a}$ other than $a_0$ are replaced by $a_i'=-a_i-2(m-1)m$, the $\bold{b}$-derived pair $\left< \bold{a'},\bold{b} \right>$ is admissible. 
\end{lemma}

\begin{proof}
By Proposition~\ref{example} the pair  $\left< \bold{a},\bold{b} \right>$ is admissible.  Let~$i$ be such that $0 \le i <m$, set $b_i'=-b_i-2m$ and note that $b_{i'} \in \{-2m^2, \ldots, 2m^2-2\}$. Then 
$$ \{-(a_j+ b_i') \colon 0 \le j < m\} = \{-1-2j+b_i+2m \colon 0 \le j < m\} = \{a_{j'}+b_i \colon 0 \le j' < m\}.$$
\noindent
Therefore, an $\bold{a}$-derived pair $\left< \bold{a},\bold{b'} \right>$, obtained by replacing any number of $b_i \in \bold{b}$ by the corresponding~$b_i'$, satisfies the conditions of Proposition~\ref{lemma1} and is thus admissible. An analogous proof for $\bold{b}$-derived pairs $\left< \bold{a'},\bold{b} \right>$ is left to the reader.
\end{proof}

\begin{table}[h!]
\centering
\caption{The basic pair $\left< \bold{a},\bold{b} \right>$  and possible substitutions for $\bold{b}$-derived and $\bold{a}$-derived pairs according to Lemma~\ref{lemma:type1}.}
\begin{adjustbox}{width=1\textwidth}
\label{table:type1}
\begin{tabular}{|c|C{0.5cm}|*{5}{c|}}
\hline
   $i$ & 0 & 1 & \ldots & i & $ \ldots$ & $m-1$ \\ \hline\hline
   $\bold{a}, \bold{a'}$ & 1 & 3, $-3-2(m-1)m$ & \ldots & $1+2i$, $-1-2i-2(m-1)m$ & \ldots & $2m-1$, $1-2m^2$ \\ \hline
   $\bold{b}, \bold{b'}$ & 0 & $2m$, $-4m$ & \ldots & $2im$, $-2im-2m$ & \ldots & $2(m-1)m$, $-2m^2$ \\ \hline
\end{tabular}
\end{adjustbox}
\end{table}

\begin{lemma} 
\label{lemma:type2}
Let $m \ge 3$ be an odd integer and let 
\setlength{\belowdisplayskip}{7pt}  \setlength{\belowdisplayshortskip}{7pt} 
\setlength{\abovedisplayskip}{7pt}  \setlength{\abovedisplayshortskip}{-3pt}
$$\bold{a} = [(-1)^i(1+2i) \colon 0 \le i < m] \quad \textrm{and} \quad \bold{b} = [(-1)^i(1+2i)m-m \colon 0 \le i < m].$$ 
Then the pair $\left< \bold{a},\bold{b} \right>$ is admissible. Moreover, if any number of elements $b_i \in \bold{b}$ other than $b_0$ are replaced by $b_i'=-b_i-2$, the $\bold{a}$-derived pair $\left< \bold{a},\bold{b'} \right>$ is admissible. Similarly, if any number of elements $a_i \in \bold{a}$ other than $a_0$ are replaced by $a_i'=-a_i$, the $\bold{b}$-derived pair $\left< \bold{a'},\bold{b} \right>$ is admissible. 
\end{lemma}

\begin{proof}
Note that $\bold{b}=[0, -4m, 4m, -8m, 8m, \ldots, -2(m-1)m, 2(m-1)m]$ (recall that $m$ is odd). It is now easy to verify that the pair  $\left< \bold{a},\bold{b} \right>$ satisfies the conditions of Proposition~\ref{lemma1} and is thus admissible. Moreover, this also shows that all $\bold{b}$-derived pairs clearly satisfy the conditions of Proposition~\ref{lemma1} and are thus admissible. For $\bold{a}$-derived pairs, let~$i$ be such that $0 \le i <m$ and set $b_i'=-b_i-2$. Then 
$$ \{-(a_j+ b_i') \colon 0 \le j < m\} = \{(-1)^j(-1-2j)+b_i+2 \colon 0 \le j < m\}.$$
\noindent
For $j=0$ we have that $-(a_0+b_i')=a_0+b_i$. For $j > 0$ set $j'=j-(-1)^j$, note that $j$ and $j'$ have different parity and that (since $m$ is odd) $\{j' \colon 1 \le j < m\}=\{j \colon 1 \le j < m\}$. Thus
\begin{equation*}
\begin{split}
\{-(a_j+ b_i') \colon 1 \le j < m\} & = \{(-1)^{j'}(1+2j'-2(-1)^{j'})+b_i+2 \colon 1 \le j' < m\} \\
						& =  \{a_{j'}+ b_i \colon 1 \le j' < m\}.
\end{split}
\end{equation*}
Therefore, an $\bold{a}$-derived pair $\left< \bold{a},\bold{b'} \right>$, obtained by replacing any number of $b_i \in \bold{b}$ by the corresponding~$b_i'$, satisfies the conditions of Proposition~\ref{lemma1} and is thus admissible.
\end{proof}

\begin{table}[h!]
\centering
\caption{The basic pair $\left< \bold{a},\bold{b} \right>$ and possible substitutions for $\bold{b}$-derived and $\bold{a}$-derived pairs according to Lemma~\ref{lemma:type2}.}
\begin{adjustbox}{width=1\textwidth}
\label{table:type2}
\begin{tabular}{|c|C{0.5cm}|*{5}{c|}}
\hline
   $i$ & 0 & 1 & \ldots & i & $ \ldots$ & $m-1$ \\ \hline\hline
   $\bold{a}, \bold{a'}$ & 1 & $-3$,3 & \ldots & $(-1)^i(1+2i)$, $(-1)^{i+1}(1+2i)$ & \ldots & $2m-1$, $1-2m$ \\ \hline
   $\bold{b}, \bold{b'}$ & 0 & $-4m$, $4m-2$ & \ldots & $(-1)^i(1+2i)m-m$, $(-1)^{i+1}(1+2i)m+m-2$ & \ldots & $2(m-1)m$, $2(1-m)m-2$ \\ \hline
\end{tabular}
\end{adjustbox}
\end{table}

Table~\ref{table:type1} and Table~\ref{table:type2} show how one can construct the basic and all derived pairs from Lemma~\ref{lemma:type1} and from Lemma~\ref{lemma:type2}, respectively. In particular, each sequence in the basic pair starts with 0 or 1, taking the leftmost element in the column for each subsequent entry while each sequence in a derived pair starts with 0 or 1, taking the rightmost element in the column for at least one subsequent entry.

\begin{proposition} \label{lemma4}
Let $m \ge 3$ be an odd integer and let $\Gamma= C_m \square C_{2m}$. Then there are at least $4m^2 \cdot (2^{m+1}-3) \cdot (m-1)!^2$ different distance magic labelings of $\Gamma$.
\end{proposition}

\begin{proof}
Since $\Gamma$ is vertex-transitive there are $2m^2$ choices for the vertex labeled with~$1$ and each such choice gives the same number of different distance magic labelings of~$\Gamma$. By Theorem~\ref{lemma2} each distance magic labeling $\ell$ of~$\Gamma$ with $\ell_{0,0}=1$ corresponds to (a unique) admissible pair $\left<\bold{a},\bold{b}\right>$ of sequences. By Proposition~\ref{lemma1} any permutation of the elements $a_1, a_2, \ldots, a_{m-1}$ in~$\bold{a}$ and $b_1, b_2, \ldots, b_{m-1}$ in~$\bold{b}$ yields another admissible pair of sequences, and by Theorem~\ref{lemma2} this pair gives a different distance magic labeling of~$\Gamma$. Therefore, each admissible pair of sets gives rise to $(m-1)!^2$ different distance magic labelings of~$\Gamma$, and so it suffices to count the number of admissible pairs of sets.

In what follows we establish a lower bound on this number by considering only admissible pairs of sets arising from those given by Lemma~\ref{lemma:type1} and Lemma~\ref{lemma:type2} and possibly applying Lemma~\ref{lemma:tilde}. Consider the pairs derived from the basic pair from Lemma~\ref{lemma:type1} (making corresponding derived pairs and possibly applying Lemma~\ref{lemma:tilde} - we call this process a {\em derivation}). There are clearly $2(1+2(2^{m-1}-1))=2(2^m-1)$ such derivations (including the trivial one), since we have to decide whether we apply Lemma~\ref{lemma:tilde} or not, decide whether we are replacing any elements in the two sets and if so, in which and which ones. Of course, there are also $2(2^m-1)$ derivations starting from the basic pair from Lemma~\ref{lemma:type2}. We thus only need to determine whether two different derivations can result in the same admissible pair. 

Suppose first that a pair $\langle\tilde{\bold{a}}, \tilde{\bold{b}}\rangle$ is obtained by a derivation starting from the basic pair $\langle \bold{a}, \bold{b}\rangle$ from Lemma~\ref{lemma:type1}. If this derivation is trivial or arises only by applying  Lemma~\ref{lemma:tilde}, the sequences $\tilde{\bold{a}}$ and $\tilde{\bold{b}}$ contain no negative elements. In all other cases, one of the sequences $\tilde{\bold{a}}$ and $\tilde{\bold{b}}$ contains at least one negative element while the other does not. Moreover, the elements of the latter one are either all smaller than $2m$ (which happens if and only if we are changing the elements of $\bold{b}$), or are (with the exception of the smallest one) all greater than $2m-1$ (which happens if and only if we are changing the elements of $\bold{a}$). It is therefore easy to see that no two derivations starting from the basic pair $\langle \bold{a}, \bold{b}\rangle$ from Lemma~\ref{lemma:type1} lead to the same admissible pair. 

Suppose next that a pair $\langle\tilde{\bold{a}}, \tilde{\bold{b}}\rangle$ is obtained by a derivation starting from the basic pair $\langle \bold{a}, \bold{b}\rangle$ from  Lemma~\ref{lemma:type2} and note that at least one of the sequences $\tilde{\bold{a}}$ and $\tilde{\bold{b}}$ contains exactly $\frac{m-1}{2}$ negative elements. If this derivation is trivial or arises only by applying  Lemma~\ref{lemma:tilde}, then no element of any of the sequences $\tilde{\bold{a}}$ and $\tilde{\bold{b}}$ is congruent to $2$ or $3$ modulo $4$. In all other cases, one of the sequences $\tilde{\bold{a}}$ and $\tilde{\bold{b}}$ contains at least one element which is congruent to $2$ or $3$ modulo $4$, while the elements of the other one are either all congruent to $1$ modulo $4$, or are all divisible by $4$. Moreover, the elements of the latter one either all have absolute value smaller than $2m$ (which happens if and only if we are changing the elements of $\bold{b}$), or all but one have absolute value at least $4m-1$ (which happens if and only if we are changing the elements of $\bold{a}$). It is therefore easy to see that no two derivations starting from the basic pair $\langle \bold{a}, \bold{b}\rangle$ from Lemma~\ref{lemma:type2} lead to the same admissible pair. 

The only possible duplications can therefore arise from a derivation of the basic pair from  Lemma~\ref{lemma:type1} and from a derivation of the basic pair from  Lemma~\ref{lemma:type2}. Let $\langle\tilde{\bold{a}}, \tilde{\bold{b}}\rangle$ be a duplicated pair. Since on one hand at least one of $\tilde{\bold{a}}$ and $\tilde{\bold{b}}$ contains no negative elements, while on the other hand at least one of $\tilde{\bold{a}}$ and $\tilde{\bold{b}}$ contains exactly $\frac{m-1}{2}$ negative elements, we find that one of $\tilde{\bold{a}}$ and $\tilde{\bold{b}}$ contains no negative elements, while the other one has precisely $\frac{m-1}{2}$ negative elements. Suppose that the elements of the sequence with no negative elements are (with one exception) all greater than $2m-1$. Since $\langle\tilde{\bold{a}},\tilde{\bold{b}}\rangle$ is obtained by a derivation starting from the basic pair from  Lemma~\ref{lemma:type2}, the other sequence must be one of $[1, -3, 5, -7, \ldots, 2m-1]$ and $[0, -4, 4, -8, \ldots, 2m-2]$. However, such a sequence clearly cannot be obtained by a derivation starting from the basic pair from  Lemma~\ref{lemma:type1}. It thus follows that either ${\bold{\tilde{a}}} = [1,3,5,\ldots , 2m-1]$ or ${\bold{\tilde{b}}} = [0,2,4,\ldots , 2m-2]$. Since $\langle\tilde{\bold{a}},\tilde{\bold{b}}\rangle$ is obtained by a derivation starting from the basic pair from  Lemma~\ref{lemma:type2}, we must have that either ${\bold{\tilde{b}}} = [0, -4m, 4m, -8m, 8m, \ldots, 2(m-1)m]$ or ${\bold{\tilde{a}}} =  [1, -4m+1, 4m+1, -8m+1, 8m+1, \ldots, 2(m-1)m+1]$, respectively. It is easy to see that these two corresponding pairs are indeed duplicated, thus showing that all derivations give precisely $4(2^m-1)-2 = 2(2^{m+1}-3)$ admissible pairs of sets.
Together with the remarks from the first paragraph of this proof, we thus have at least $4m^2 \cdot (2^{m+1}-3) \cdot (m-1)!^2$ different distance magic labelings of~$\Gamma$.
\end{proof}

Despite the fact that we could not determine the exact number of all distance magic labelings of $C_m \square C_{2m}$ for each $m \ge 3$ odd, the result from Proposition~\ref{lemma4} seems to give a good lower bound, especially for larger values of~$m$. In fact, a computer search reveals that for $m \in \{3,5,7\}$ the ratio of the bound from Proposition~\ref{lemma4} and the exact number of distance magic labelings is $26/34 = 0.765$, $122/148 = 0.824$ and $506/538 = 0.941$, respectively.

\section{Distance magic labelings of $\bold{C_{2m} \square C_{2m}}$ with $\bold{m \ge 3}$ odd} \label{sec:4}

In this section we analyze distance magic labelings for the graphs $C_{2m} \square C_{2m}$ with $m \ge 3$ odd. It turns out that all distance magic labelings of these graphs can be obtained in a similar way as was done in the previous section. Before we proceed we introduce some additional notation that we will be using in this section.
Let $m \ge 3$ be an odd integer and let $L=[\ell_{i,j}]_{0 \le i,j< 2m}$ be a $2m \times 2m$ table of integers. Then the {\em pair of partial tables}, $par(L)$, of $L$ is the pair $\langle T, T'\rangle$ of $m \times m$ tables $T=[t_{i,j}]_{0 \le i,j < m}$ and $T'=[t'_{i,j}]_{0 \le i,j < m}$, where
\begin{equation} \label{eq:min}
t_{i,j}=\ell_{2i, 2j} \quad \textrm{and} \quad t'_{i,j}=\ell_{2i+1, 2j}, \textrm{ } 0 \le i, j < m.
\end{equation}

\noindent
Conversely, let $T=[t_{i,j}]_{0 \le i,j < m}$ and $T'=[t'_{i,j}]_{0 \le i,j < m}$ be two $m \times m $ tables of integers. Then the corresponding {\em merged table}, $mer(T,T')$, of $T$ and $T'$ is the $2m \times 2m$ table $[\ell_{i,j}]_{0 \le i,j < 2m}$, where
\begin{equation} \label{eq:asdfgh}
\begin{split}
\ell_{i,j}=\left\{
\begin{array}{ll}
t_{\frac{i}{2}, \frac{j}{2}}, & i,j \textrm{ even} \\
-t_{\frac{i+m}{2}, \frac{j+m}{2}}, & i,j \textrm{ odd} \\
t'_{\frac{i-1}{2}, \frac{j}{2}}, & i \textrm{ odd }, j \textrm{ even} \\
-t'_{\frac{i+m-1}{2}, \frac{j+m}{2}}, & i \textrm{ even}, j \textrm{ odd}.
\end{array}
\right.
\end{split}
\end{equation}
Note that since $m$ is odd, the above definition does indeed make sense and implies that $\ell_{i,j}=-\ell_{i+m, j+m}$ for all $i,j$ with $0 \le i,j < 2m$ (where the indices are computed modulo $2m$).

As in Section \ref{sec:3} this leads to the following definition. Let $m \ge 3$ be an odd integer. A pair $\langle T, T' \rangle$ of $m \times m$ tables of integers $T=[t_{i,j}]_{0 \le i,j < m}$ and $T'=[t'_{i,j}]_{0 \le i,j < m}$ is said to be {\em distance magic} if for each $k \in \mathcal{N}_{4m^2}$ exactly one of $k, -k$ appears in any of $T$ and $T'$, and for all $i,j$ with $0 \le i,j <m$, we have that
\begin{equation} \label{eq:tt}
\begin{split}
 t_{i-1, j} + t_{i, j} &= t_{i+ \frac{m-1}{2}, j+\frac{m-1}{2}} + t_{i+\frac{m-1}{2}, j+\frac{m+1}{2}}, \\
 t'_{i-1, j} + t'_{i, j} &= t'_{i+ \frac{m-1}{2}, j+\frac{m-1}{2}} + t'_{i+\frac{m-1}{2}, j+\frac{m+1}{2}}.
\end{split}
\end{equation}

\begin{lemma} \label{??}
Let $m \ge 3$ be an odd integer and let $\Gamma = C_{2m} \square C_{2m}$. If $\Gamma$ is distance magic and $L$ denotes the $2m \times 2m$ table corresponding to a distance magic labeling of $\Gamma$, then the pair of tables $par(L)$ is distance magic. Conversely, if the pair $\langle T, T' \rangle$ of $m \times m$ tables $T$ and $T'$ is distance magic then $mer(T, T')$ corresponds to a distance magic labeling of~$\Gamma$ and consequently $\Gamma$ is distance magic.
\end{lemma}

\begin{proof}
Suppose that $\Gamma$ is distance magic and let $L$ denote the $2m \times 2m$ table corresponding to a distance magic labeling~$\ell$ of~$\Gamma$. Let $\langle T, T'\rangle = par(L)$. The table~$L$ consists of the $4m^2$ integers from $\mathcal{N}_{4m^2}$. By Corollary~\ref{lemma:submm} we have that $\ell_{i,j}= -\ell_{i+m, j+m}$ for all $i, j$ with $0 \le i,j < 2m$, and so (\ref{eq:min}) implies that for each $k \in \mathcal{N}_{4m^2}$ exactly one of $k, -k$ appears in any of $T$ and  $T'$. Moreover, since $\ell$ is a distance magic labeling, (\ref{eq:min}) and Corollary~\ref{lemma:submm} imply that for all $i,j$ with $0 \le i,j < 2m$ we have that
\begin{equation*} 
\begin{split}
t_{i-1, j} + t_{i, j} & = \ell_{2i-2, 2j} + \ell_{2i, 2j} = -\ell_{2i-1, 2j-1} - \ell_{2i-1, 2j+1} \\
 & = \ell_{2i+m-1, 2j+m-1} + \ell_{2i+m-1, 2j+m+1} = t_{i+ \frac{m-1}{2}, j+\frac{m-1}{2}} + t_{i+\frac{m-1}{2}, j+\frac{m+1}{2}}
\end{split}
\end{equation*} and
$$ t'_{i-1, j} + t'_{i, j} = t'_{i+\frac{m-1}{2}, j+\frac{m-1}{2}} + t'_{i+\frac{m-1}{2}, j+\frac{m+1}{2}}.$$
Therefore, the pair $\langle T, T' \rangle$ is distance magic.

For the converse, suppose that the pair $\langle T, T' \rangle$ of tables $T$ and $T'$ is distance magic and let $L=mer(T,T')$. Then (\ref{eq:asdfgh}) implies that $L$ consists of all the integers from $\mathcal{N}_{4m^2}$. Moreover, it is easy to verify that by (\ref{eq:asdfgh}) and (\ref{eq:tt}) we have that $\ell_{i-1, j} + \ell_{i+1, j} + \ell_{i, j-1} + \ell_{i, j+1} = 0 $ for all $i,j$ with $0 \le i, j <2m$, implying that $L$ corresponds to a distance magic labeling of~$\Gamma$.

\end{proof}

\begin{proposition} \label{lemma21}
Let $m \ge 3$ be an odd integer and let $\Gamma=C_{2m} \square C_{2m}$.  Let $a_0, a_1, \ldots, a_{m-1} \in \{1-4m^2, \ldots, 4m^2-1\}$ and $c_0, c_1, \ldots, c_{m-1} \in \{1, \ldots, 8m^2-1\}$ be odd integers, and $b_0, b_1, \ldots, b_{m-1} \in \{-4m^2, \ldots, 4m^2-2\}$ and $d_0, d_1, \ldots, d_{m-1} \in \{-4m^2, \ldots, 4m^2-2m\}$ be even integers such that $a_0=c_0=1$, $b_0=0$ and $\{x(a_i+b_j) \colon x \in \{-1, 1\}, 0 \le i,j <m  \} \cup \{x(c_i+d_j) \colon x \in \{-1, 1\}, 0 \le i,j <m  \} = \mathcal{N}_{4m^2}$. Then setting $T= \mathcal{T}([a_0, a_1, \ldots, a_{m-1}], 1)+ \mathcal{T}([b_0, b_1, \ldots, b_{m-1}], m-1)$ and $T'= \mathcal{T}([c_0, c_1, \ldots, c_{m-1}], 1)+ \mathcal{T}([d_0, d_1, \ldots, d_{m-1}], m-1)$, the pair $\langle T, T' \rangle$ is distance magic and thus $\Gamma$ is distance magic.
\end{proposition}

\begin{proof}
Let $A=\mathcal{T}([a_0, a_1, \ldots, a_{m-1}], 1)$ and $B=\mathcal{T}([b_0, b_1, \ldots, b_{m-1}], m-1)$. For any $i, j$ with $0 \le i, j < m$ we have that
$$a_{i-1,j}=a_{j-i+1}=a_{i+\frac{m-1}{2}, j+\frac{m+1}{2}}, \quad a_{i,j}=a_{j-i}=a_{i+\frac{m-1}{2}, j+\frac{m-1}{2}} \quad \textrm{and}$$  
$$b_{i-1,j}=b_{j+i-1}=b_{i+\frac{m-1}{2}, j+\frac{m-1}{2}}, \quad b_{i,j}=b_{j+i}=b_{i+\frac{m-1}{2}, j+\frac{m+1}{2}},$$
which implies that the condition from (\ref{eq:tt}) holds for the table~$T$ (recall that $t_{i,j}=a_{i,j}+b_{i,j}$). Since $T'$ is constructed in the same manner, the condition from (\ref{eq:tt}) also holds for~$T'$.

We now prove that for each $k \in \mathcal{N}_{4m^2}$ exactly one of $k, -k$ appears in any of $T$ and $T'$. To do so it clearly suffices to prove that $\{t_{i,j} \colon 0 \le i,j <m \} = \{a_i+b_j \colon 0 \le i,j <m \}$ and $\{t'_{i,j} \colon 0 \le i,j <m \} = \{c_i+d_j \colon 0 \le i,j <m \}$. We verify only the first equality (the second one is analogous). Clearly, $\{t_{i,j} \colon 0 \le i,j <m \} \subseteq \{a_i+b_j \colon 0 \le i,j <m \}$. Conversely, for each $i,j$ with $0 \le i, j < m$, it is easy to verify that letting 
\begin{multicols}{2}
\noindent
\begin{equation*}
\begin{split}
i' = \left\{
\begin{array}{ll}
\frac{j-i}{2}, & \modulo{i}{j}{2} \\
\frac{j-i+m}{2}, & \notmodulo{i}{j}{2}
\end{array}
\right. \quad \textrm{and}
\end{split}
\end{equation*}
\begin{equation*}
\begin{split}
j' = \left\{
\begin{array}{ll}
\frac{j+i}{2}, & \modulo{i}{j}{2} \\
\frac{j+i+m}{2}, & \notmodulo{i}{j}{2}
\end{array}
\right. \quad
\end{split}
\end{equation*}
\end{multicols}
\noindent
we have that $t_{i',j'}=a_{i}+b_{j}$. Therefore, the pair $\langle T, T' \rangle$ is distance magic. By Lemma~\ref{??} the table $mer(T,T')$ corresponds to a distance magic labeling of~$\Gamma$. 
\end{proof}

To prove that the graph $C_{2m} \square C_{2m}$ is distance magic for all $m \ge 3$ odd, it thus suffices to find sets of integers $a_i$, $b_i$, $c_i$, $d_i$ with $0 \le i <m$, satisfying the conditions of Proposition~\ref{lemma21}.

\begin{proposition}\label{example2}
Let $m\ge 3$ be an odd integer. Let 
\setlength{\belowdisplayskip}{7pt}  \setlength{\belowdisplayshortskip}{7pt} 
\setlength{\abovedisplayskip}{7pt}  \setlength{\abovedisplayshortskip}{-3pt}
$$\bold{a} =[1+2i \colon 0 \le i <m], \quad \bold{b}=[2im \colon 0 \le i <m], \quad \bold{d}=[2im+2m^2 \colon 0 \le i < m],$$ $$T= \mathcal{T}(\bold{a}, 1)+ \mathcal{T}(\bold{b}, m-1) \quad {\rm and} \quad T'= \mathcal{T}(\bold{a}, 1)+ \mathcal{T}(\bold{d}, m-1).$$ Then the pair  $\langle T, T' \rangle$ is distance magic and thus the graph $C_{2m} \square C_{2m}$ is distance magic.
\end{proposition}

\begin{proof}
Clearly, $\{x(a_i+b_j) \colon x \in \{-1, 1\}, 0 \le i, j <m\} = \mathcal{N}_{2m^2}$ and $\{x(a_i+d_j) \colon x \in \{-1, 1\}, 0 \le i, j <m\} = \mathcal{N}_{4m^2} \setminus \mathcal{N}_{2m^2}$. Therefore Proposition~\ref{lemma21} implies that the pair $\langle T, T' \rangle$ is distance magic and so the graph $C_{2m} \square C_{2m}$ is distance magic (see Figure~\ref{fig:Case2} for a corresponding distance magic labeling of $C_{14} \square C_{14}$). 
\end{proof}

\begin{figure}[h] 
\centering
\includegraphics[scale=0.75, trim={1.7cm 16.4cm 3cm 1.8cm}, clip]{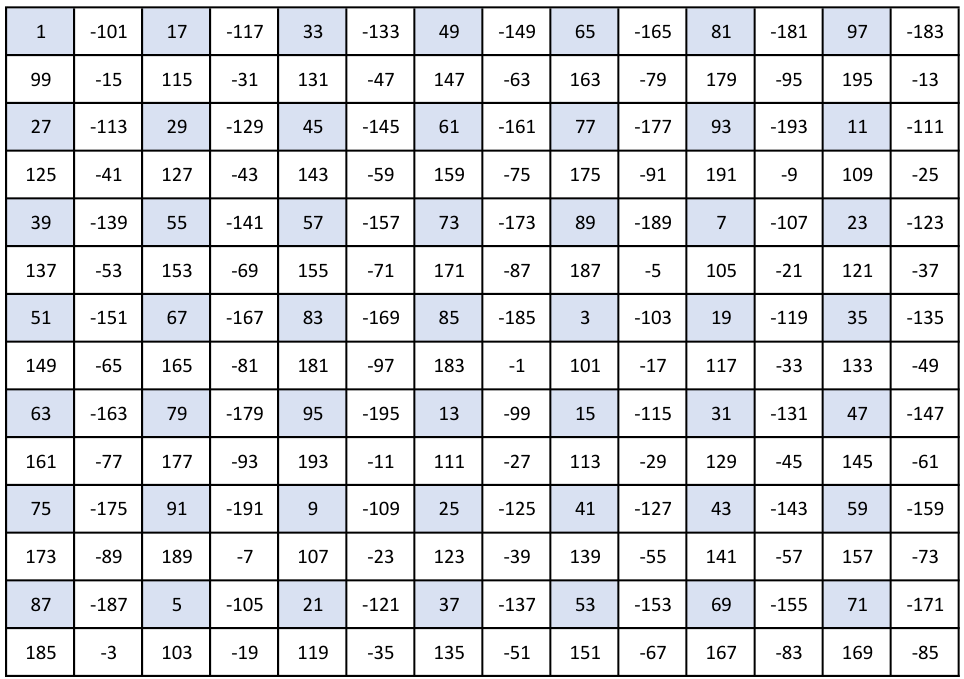}
\caption{A table corresponding to a distance magic labeling of $C_{14} \square C_{14}$.}
\label{fig:Case2}
\end{figure}

Having proved the existence of a distance magic labeling of $C_{2m} \square C_{2m}$ for each odd $m \ge 3$, we now address the problem from~\cite{RSP04} of determining all distance magic labelings of these graphs. We first make the following simple observation whose proof is a simple induction on $s$ and is left to the reader.

\begin{lemma} \label{lemmas2}
Let $m \ge 3$ be an odd integer and let $\ell$ be a distance magic labeling of $C_{2m} \square C_{2m}$. Then for all $i$, $j$, $s$ with $0 \le i,j,s  < 2m$ we 
have that
\begin{equation} \label{eq:labeling}
\ell_{i,j} + \ell_{i-1,j+1} = (-1)^s(\ell_{i+s,j+s} + \ell_{i+s-1, j+s+1}).
\end{equation}
\end{lemma}

We are now ready to prove that in fact all distance magic labelings of $C_{2m} \square C_{2m}$ with $m \ge 3$ odd arise from the construction in Proposition~\ref{lemma21}.

\begin{theorem} \label{lemma22}
Let $m\ge 3$ be an odd integer and let $\ell$ be a distance magic labeling of $C_{2m} \square C_{2m}$. If $\ell_{0,0}=1$ then there exist uniquely determined sequences $\bold{a}=[a_i \colon 0 \le i < m]$, $\bold{b}=[b_i \colon 0 \le i < m]$, $\bold{c}=[c_i \colon 0 \le i < m]$ and $\bold{d}=[d_i \colon 0 \le i < m]$ satisfying the assumptions of Proposition~\ref{lemma21}, such that $\ell$ is the labeling obtained from them using the construction in Proposition~\ref{lemma21}.
\end{theorem}

\begin{proof}
Let $L$ denote the $2m \times 2m$ distance magic table corresponding to~$\ell$. Set 
\begin{equation*}
\bold{a} = [\ell_{i(m-1), i(m+1)} \colon 0 \le i < m],
\quad\quad
\bold{b} = [\ell_{i(m+1), i(m+1)}-1 \colon 0 \le i < m],
\end{equation*}
\begin{equation*}
\bold{c} = [\ell_{i(m-1)+1, i(m+1)}-\ell_{1,0}+1 \colon 0 \le i < m],
\quad\quad
\bold{d} = [\ell_{i(m+1)+1, i(m+1)}-1 \colon 0 \le i < m],
\end{equation*}
and let 
\begin{equation*}
A=\mathcal{T}(\bold{a},1), \quad B= \mathcal{T}({\bold{b},m-1}), \quad C=\mathcal{T}(\bold{c},1), \quad {\rm and} \quad D= \mathcal{T}({\bold{d},m-1}).
\end{equation*}
We claim that $mer( A+B, C+D ) =L$, that is, $\ell_{2i,2j} = a_{i,j}+b_{i,j}=-\ell_{2i+m, 2j+m}$ and $\ell_{2i+1, 2j}=c_{i,j}+d_{i,j}=-\ell_{2i+m+1, 2j+m}$ for all $i,j$ with $0 \le i,j <m$.
Note that we only need to show that
\begin{equation}\label{eq:noidea}
\ell_{2i,2j} = a_{i,j}+b_{i,j} \quad {\rm and } \quad \ell_{2i+1,2j} = c_{i,j}+d_{i,j}.
\end{equation}
We prove only the first part of~(\ref{eq:noidea}) (the proof of the second part is similar and is left to the reader). It clearly suffices to show that $a_{i,i+k}+b_{i,i+k}=\ell_{2i,2i+2k}$ for all $i,k$ with $0 \le i,k <m$ which is equivalent to 
\begin{equation} \label{eq:ind}
\ell_{k(m-1), k(m+1)} + \ell_{(2i+k)(m+1), (2i+k)(m+1)} -1 = \ell_{2i, 2i+2k}.
\end{equation}
Proceeding by induction on~$k$ and applying Lemma~\ref{lemmas2} we find that (\ref{eq:ind}) indeed holds for each~$k$ (we leave the technical details to the reader).  
All that is left to verify is the uniqueness of the sequences $\bold{a}$, $\bold{b}$, $\bold{c}$ and $\bold{d}$ satisfying the assumptions stated in Proposition~\ref{lemma21}. This follows from $a_0=c_0=1$ by an argument similar to the one in the proof of Theorem~\ref{lemma2}.  
\end{proof}

\section{Concluding remarks and future research}

It follows from Proposition~\ref{lemma:nec} that the condition from Theorem~\ref{theorem1} is necessary for a Cartesian product of cycles to be distance magic, while the fact that it is sufficient follows from Proposition~\ref{example} and Proposition~\ref{example2}. Thus, Theorem \ref{theorem1} gives a complete classification of distance magic Cartesian products of cycles.

As for the problem from \cite{RSP04} of determining all distance magic labelings of such graphs we thus far have the following results. In Theorem~\ref{lemma2} and Theorem~\ref{lemma22} we have shown that each distance magic labeling of a Cartesian product of cycles arises from a pair or quadruple of suitable sequences of integers. Consequently, finding all distance magic labelings of such graphs is equivalent to identifying all such pairs or quadruples of suitable sequences. Moreover, in Proposition~\ref{lemma4} we established a lower bound on the number of all suitable pairs for $C_{m} \square C_{2m}$ with $m \ge 3$ odd by counting the pairs of suitable sequences which arise in a ``natural" way (as described in Proposition~\ref{lemma4} and the series of lemmas preceding it). We did not provide an analogous lower bound on the number of all suitable quadruples arising in a ``natural" way for the graphs $C_{2m} \square C_{2m}$ for $m \ge 3$ odd, but we anticipate that an approach similar to the one used for $C_{m} \square C_{2m}$ can be used in this case as well. The following two problems therefore present a natural continuation of the research undertaken in this paper.

\begin{problem}
Determine the number of all pairs of sequences satisfying Proposition~\ref{lemma1}.
\end{problem}

\begin{problem}
Determine the number of all quadruples of sequences satisfying Proposition~\ref{lemma21}.
\end{problem}

We mention another possible avenue for future research. Recall that the Cartesian product $C_{m} \square C_{n}$ of cycles is a very special example of a tetravalent Cayley graph of an abelian group, namely it is $\Cay(\ZZ_m \times \ZZ_n; \{\pm(1,0), \pm(0,1)\} )$. In \cite{MS21}, where all distance magic tetravalent Cayley graphs of cyclic groups were classified, the problem of determining all distance magic tetravalent Cayley graphs of abelian groups was posed (see \cite[Problem 5.1]{MS21}). Theorem \ref{theorem1} thus presents the first step towards the solution of this problem. We are convinced that the approach using group characters from Section \ref{sec:2} can be used to analyze the distance magic property of tetravalent Cayley graphs of abelian groups with respect to more general connection sets.

\section*{Statements and Declarations}
Funding: \\
This work was supported in part by the Slovenian Research and Innovation Agency (Young researchers program, research program P1-0285 and research projects J1-2451, J1-3001 and J1-50000).
\medskip

\noindent
Competing interests: \\
The authors declare that they have no known competing financial interests or personal relationships that could have appeared to influence the work reported in this paper.

\printbibliography

\end{document}